\documentclass[twoside,11pt]{amsart}
	
\usepackage{amsmath,latexsym,amssymb,mathptm, verbatim, enumerate, ytableau}
\usepackage{amsfonts,epsfig,textcomp,mathdots}

\usepackage{mathtools, soul}
\usepackage{eucal}
\usepackage{tikz}
\usepackage{tikz-cd}
\usetikzlibrary{%
  matrix,%
  calc,%
  arrows%
}
\usepackage{dsfont}
\usetikzlibrary{matrix}

\usepackage[all,cmtip]{xy}

\input amssym.def
\input xypic
\input xy

\xyoption{all}
\def\S\EDcortwo{5.5}

\def\sqr#1#2{{\vcenter{\hrule height.#2pt
        \hbox{\vrule width.#2pt height#1pt \kern#1pt
                \vrule width.#2pt}
        \hrule height.#2pt}}}

\def\q{\mathfrak q}
\def\p{\mathfrak p}

\def\P{\mathbb P}

\def\m{\mathfrak m}

\def\Min{\operatorname{Min}}

\def\p{\mathfrak p}
\def\m{\mathfrak m}

\def\n{\mathfrak n}

\def\dim{{\rm dim \, }}

\def\bs{\bigskip}

\def\G{ G}

\newtheorem{theorem}{Theorem}[section]
\newtheorem{lemma}[theorem]{Lemma}
\newtheorem{corollary}[theorem]{Corollary}
\newtheorem{proposition}[theorem]{Proposition}

\newtheorem{def&dis}[theorem]{Definition and Discussion}
\theoremstyle{definition}

\newtheorem{remark}[theorem]{Remark}

\newtheorem{conjecture}[theorem]{Conjecture}

\newtheorem{example}[theorem]{Example}

\newtheorem{problem}[theorem]{Problem}

\begin{document}

\baselineskip=15pt

\title[Multiplicity sequence and integral dependence]
{Multiplicity sequence and integral dependence}

\author[C. Polini, N. V. Trung, B. Ulrich, \and J. Validashti]
{Claudia Polini, Ngo Viet Trung, Bernd Ulrich, \and Javid Validashti}

\thanks{MSC 2020 {\em Mathematics Subject Classification}.
Primary 13B22, 13D40, 13A30; Secondary 14B05, 14D99}

\keywords{Hilbert-Samuel multiplicity, multiplicity sequence, $j$-multiplicity, Segre numbers, reduction, integral dependence, principle of specialization of integral
dependence}

\thanks{The first and third authors were partially supported by NSF grants DMS-1601865 and DMS-1802383, respectively.
The second author was partially supported by grant 101.04-2019.313 of the Vietnam National Foundation for Science and Technology Development.}

\address{Department of Mathematics, University of Notre Dame, Notre Dame, IN 46556, USA} \email{cpolini@nd.edu}
\address{International Centre for Research and Postgraduate Training, Institute of Mathematics, Vietnam Academy of Science and Technology, 10307 Hanoi, Vietnam}\email{nvtrung@math.ac.vn}
\address{Department of Mathematics, Purdue University, West Lafayette, IN 47907, USA}\email{bulrich@purdue.edu}
\address{Department of Mathematics, DePaul University, Chicago, IL 60604, USA}\email{jvalidas@depaul.edu}
\begin{abstract}  We prove that two arbitrary ideals $I \subset J$ in an equidimensional and universally catenary Noetherian local ring have the same integral closure if and only if they have the same multiplicity sequence. We also obtain a Principle of Specialization of Integral Dependence, which gives a condition for integral dependence in terms of the constancy of the multiplicity sequence in families.
\end{abstract}

\maketitle

\vspace{-0.23in}

\section{Introduction}

The aim of this paper is to prove a numerical criterion for integral dependence of arbitrary ideals, which is an important topic in commutative algebra and singularity theory. 

The first numerical criterion for integral dependence was proved by Rees in 1961 \cite{R61}: Let  $I\subset J$ be two $\m$-primary ideals in an equidimensional and universally catenary Noetherian local ring $(R, \m)$. Then $I$ and $J$ have the same integral closure if and only if they have the same Hilbert-Samuel multiplicity. This multiplicity theorem plays an important role in Teissier's work on the equisingularity of families of hypersurfaces with isolated singularities, as it is used in the proof of his principle of specialization of integral dependence (PSID) \cite{T1,T2}.  For hypersurfaces with non-isolated singularities, one needs a similar numerical criterion for integral dependence of non-$\m$-primary ideals. 

In 1969, B\"oger \cite{B1} extended Rees' multiplicity theorem to the case of equimultiple ideals.  We refer to a survey of Lipman for the geometric significance of equimultiplicity \cite{Li}. Subsequently, there were further generalizations that still maintain remnants of the $\m$-primary assumption, see for instance
\cite{R2, Mc, R85, G03}. 
For a long time, it was not clear how to extend Rees' multiplicity theorem to arbitrary ideals. 
Since the Hilbert-Samuel multiplicity is no longer defined for non-$\m$-primary ideals, the need arose to use other notions of multiplicities that can 
be used to check for integral dependence. As it turns out, there are several choices, each with its own advantages and disadvantages.

One possibility is the $j$-multiplicity, which was defined by Achilles and Manaresi  \cite{AM} as the multiplicity of the $\m$-torsion of  the associated graded ring of an ideal.
Another option is the $\varepsilon$-multiplicity, which was introduced by Ulrich and Validashti \cite{UV11} (see also \cite{KV})  to control the asymptotic behavior of the $\m$-torsion 
modulo the powers of an ideal. 
In 2001, Flenner and Manaresi \cite{FM} proved that if  $I \subset J$ are arbitrary ideals in an equidimensional and universally catenary Noetherian local ring, then $I$ and $J$ have the same integral closure if and only if they have the same $j$-multiplicity at every prime ideal or, equivalently, at every prime ideal
where $I$ has maximal analytic spread. 
An analogous statement using the $\varepsilon$-multiplicity was shown in 2010 by Katz and Validashti \cite{KV}. Both criteria require 
localization, at all prime ideals or at a finite set of prime ideals that may be difficult to determine.

On the other hand, the $j$-multiplicity of an ideal $I$ appears as one of the numbers in the multiplicity sequence, 
which consists of the normalized leading coefficients of the bivariate Hilbert polynomial of a bi-graded ring associated to $I$ and $\m$ (see Section 2 for the definition). This notion was introduced by Achilles and Manaresi \cite{AM97} and has its origin in the intersection numbers of the St\"uckrad-Vogel algorithm in intersection theory \cite{FOV}.  
It follows from the work of Achilles and Manaresi that the multiplicity sequence encodes
information about the $j$-multiplicities at the prime ideals where $I$ has maximal analytic spread. 
Ciuperca \cite[2.7]{C} showed that if the ideals $I \subset J$ have the same integral closure, then they have the same multiplicity sequence. It has since been conjectured that the converse is also true like in Rees' multiplicity theorem (see e.g. \cite[11.6]{TV}):

\begin{conjecture}\label{Q} Let $I \subset J$ be arbitrary ideals in an equidimensional and universally catenary Noetherian local ring. The ideals $I$ and $J$ have the same integral closure if and only if they have the same multiplicity sequence.
\end{conjecture}

This conjecture was inspired by
the work of Gaffney and Gassler on hypersurface singularities in 1999 \cite{GG}. 
For every reduced closed analytic subspace $(X,0) \subset ({\mathbb C}^n,0)$ of pure dimension and every ideal $I \subset {\mathcal O}_{X,0}$, they defined a set of invariants called Segre numbers, which arise from the intersection of the exceptional divisor on the blowup of $I$  with generic hyperplanes. If $I$ is the Jacobian ideal of a hypersurface singularity, the Segre numbers are just the L\^e numbers introduced by Massey in order to study equisingularity conditions \cite{M1,M2}. Later, Achilles and Rams \cite{AR} showed that the Segre numbers are a special case of the multiplicity sequence. Inspired by Teissier's work, Gaffney and Gassler \cite{GG} proved a principle of specialization of integral dependence (PSID) based on Segre numbers. The PSID says, essentially, that two ideal sheaves $\mathcal I \subset \mathcal J$ defined on the total space of a family have the same integral closure if they do so on the generic fiber and if suitable numerical invariants of $\mathcal I$ are constant across the fibers of the family. As a consequence, two ideals  $I \subset J$ of  ${\mathcal O}_{X,0}$ have the same integral closure if and only if they have the same Segre numbers. Therefore, Conjecture \ref{Q} has an affirmative answer in the analytic case. It has been a great challenge to extend the results of Gaffney and Gassler to arbitrary local rings.

In this paper we prove that two arbitrary ideals $I \subset J$ in an equidimensional and universally catenary Noetherian local ring have the same integral closure if and only if they have the same multiplicity sequence, thereby solving Conjecture ~\ref{Q} in full generality. 
The basic idea is to test integral dependence locally at the prime ideals where the ideal $I$ has maximal analytic spread. We first prove a key technical result that characterizes parameters
that belong to none of these prime ideals
in terms of the multiplicity sequence. From this we deduce both the affirmative answer to Conjecture~\ref{Q} and the PSID based on the multiplicity sequence. 
The multiplicity sequence, as opposed to the $j$-multiplicity or the $\varepsilon$-multiplicity, avoids the need to consider localizations and, most notably, it is easily computable 
like the $j$-multiplicity, using the intersection algorithm, and it behaves well in families like the $\varepsilon$-multiplicity, as it satisfies a PSID.

We could not deduce the aforementioned criterion of Flenner and Manaresi from our results. 
On the other hand, we can strengthen their criterion by showing that two arbitrary ideals $I \subset J$ in an equidimensional and universally catenary Noetherian local ring have the same integral closure if and only if they have the same refined multiplicity sequence (see Section 5 for the definition). The refined multiplicity sequence accounts for the contribution of the local $j$-multiplicities in the original multiplicity sequence.
In  St\"uckrad's and Vogel's approach to intersection theory, the refined multiplicity sequence gives the degree of the part of the intersection cycle that is supported at
the {\it rational} components of a fixed dimension; these components are the distinguished varieties in Fulton's intersection theory \cite{Ga}.

\smallskip

\section{Preliminaries}

In this section we recall some definitions and establish basic properties, mainly of the multiplicity sequence, that will be used throughout. 

General elements are instrumental in the study of multiplicities. To review the definition, 
let $R$ be a Noetherian local ring with an infinite residue field $k$ and let $I$ be an ideal of $R$ 
generated by $a_1, \ldots, a_n$. We say that $\, x_1, \ldots , x_s$ are {\it general elements} of $I$, if
$x_i=\sum_{j=1}^n \lambda_{ij} \, a_j$ for  $\lambda_{ij} \in R$ and  the image 
of $(\lambda_{ij})\in R^{sn}$ in $k^{sn}$ belongs to a given dense open subset of $k^{sn}$.

Now let $R$ be a Noetherian ring and $I$ an ideal. The {\it Rees ring} of $I$ is defined as the standard graded subalgebra ${\mathcal R}(I):=R[It] \cong
\oplus_{v\ge 0} \, I^v$ of the polynomial ring $R[t]$, the {\it associated graded ring} is $G_{I}(R) :={\mathcal R}(I)\otimes_RR/I \cong \oplus_{v\ge 0} \, I^v/I^{v+1}$,
and, if $(R,\m,k)$ is local, the {\it special fiber ring} is $F(I): ={\mathcal R}(I)\otimes_Rk \cong \oplus_{v\ge 0} \, I^v/\m I^{v}$.

Let $J$ be an ideal containing $I$. One says that $J$ is {\it integral} over $I$, or $I$ is a {\it reduction} of $J$, if 
the inclusion ${\mathcal R}(I) \subset {\mathcal R}(J)$ is an integral extension of rings, equivalently, if $J^{n+1}=IJ^n$ for
$n \gg 0$, or yet equivalently, if every element $x\in J$ satisfies an equation of the form
$$x^{n}+a_{1}x^{n-1}+ \ \ldots \ +a_{n-1}x+a_{n}=0 $$
with $a_{i}\in I^{i}$ for $1 \leq i \leq n$. Of particular importance for us is the fact, which is obvious from the second characterization of integral
dependence, that if $I$ is zero-dimensional and a reduction of $J$, then the equality of Hilbert-Samuel multiplicities $e(I,R) = e(J,R)$
obtains.

If $(R, \m, k)$ is a Noetherian local ring of dimension $d$, then the {\it analytic spread} of an ideal $I$ is defined as $\ell(I):=\dim \, F(I)$.
The analytic spread of a proper ideal satisfies the inequality ${\rm ht} \, I \le \ell(I) \le d $, in particular, $\ell(I)=d$ if $I$ is $\m$-primary. Moreover, 
$\ell(I)=0$ if and only if $I$ is nilpotent.
Every ideal $I$ has a {\it minimal reduction}, a reduction 
minimal with respect to inclusion. If $k$ is infinite, then all minimal reductions of $I$ have the same minimal number 
of generators, namely $\ell(I)$; this follows from the fact that a sequence of elements in $I$ minimally generates a minimal reduction of $I$ if 
and only if its image in $I/\m I$ forms a system of parameters of $F(I)$, a ring of dimension $\ell(I)$.  Thus one also sees that
$\ell(I)$ general elements of $I$ generate a minimal reduction of $I$ and, in particular, that $d$ general elements of $I$ generate a reduction. 
The notion of minimal reduction and its relationship to multiplicities as well as analytic spread is due to Northcott and Rees;  we refer 
to \cite{NR} or \cite{SH} for more details.

Again, let $(R,\m, k)$ be a Noetherian local ring of dimension $d$ and $I$ an ideal. Consider the doubly associated graded ring 
$$\G := G_{\m}(G_{I}(R)) = \bigoplus_{u \ge 0} \,  \m^uG_I(R)/\m^{u+1} G_I(R) \, .$$
This is a Noetherian standard bigraded $k$-algebra 
with bigraded components
$$\G_{uv}=\frac{\m^uI^v+I^{v+1}}{\m^{u+1}I^v+I^{v+1}}\, .$$
Let 
$$h(r,s)=\sum_{u=0}^{r} \sum_{v=0}^{s} \lambda(\G_{uv})\, ,$$
where $\lambda(\cdot)$ denotes length. 
It is well-known that for $r$ and $s$ sufficiently large, $h(r,s)$ is a polynomial function of degree at most $d$ (equal $d$ if $I\neq R$) of the form 
$$ \sum_{i=0}^d \,  \frac{c_i (\G)}{(d-i)! \, i!} \, r^{d-i}s^{i}  +\mbox{  terms of lower degree , }
$$
where $c_i(\G)$ are nonnegative integers. 

The {\it multiplicity sequence} of the ideal $I$ is defined by Achilles and Manaresi \cite{AM97} as the sequence 
$$
c_i(I)=c_i(I,R):=c_i(\G)  \qquad  {\rm for} \  \ 0 \leq i \leq d\, .
$$
We set $c_i(I):=0$ for $i<0$ and $i>d$. 
The reader should be warned that our definition is slightly different from the one of Achilles and Manaresi in the sense that we index the sequence using codimension rather than dimension. 

From \cite[2.3(i)]{AM97} or Proposition \ref{Lformula} below we know that
\begin{equation*}\label{vanishing}c_i(I)=0 \ \ \ \ \ \  \hbox{if} \ \ i< d - \dim \, R/I  \ \ \hbox{or} \ \ i >\ell(I) \, .\end{equation*}
If $I$ is an $\m$-primary ideal, then $c_i(I) = 0$ for $i < d$ and $c_d(I) = e(I,R)$, the Hilbert-Samuel multiplicity of $I$ \cite[2.4(i)]{AM97}.  
 
For $c_0(I)$ one has the formula 
\begin{equation}\label{0}c_0(I)=\displaystyle 
\sum_{\p\in V(I), \ \dim R/\p = d}  \lambda \left(R_{\p}\right) \cdot e(R/\p) \, , \end{equation}
see \cite[2.3(iii)]{AM97} or Proposition \ref{Lformula}.
For $i\ge 1$ one can compute $c_i(I)$ using general elements. The next proposition gives the relevant formula, which
was proved by Achilles and Manaresi \cite[4.1]{AM97}. 
This formula has its origin in the St\"uckrad--Vogel algorithm in intersection theory \cite{FOV}. 

\begin{proposition}[\bf Length Formula for Segre Numbers]\label{Lformula}  
Let $R $ be a Noetherian local ring  of dimension $d$ with infinite residue field and $I$ an ideal.  If $i\ge 0$ and $x_1, \ldots, x_i$ are general elements of $I$, then  \begin{equation*} 
c_i(I)=\sum_{\substack{\p\in V(I), \ \dim R/\p = d-i\\
\p \supset (x_1, \ldots, \, x_{i-1}):\, I^{\infty} }
}  \lambda \left(\frac{R_{\p}}{ (x_1, \ldots,  x_{i-1})R_{\p}: I^{\infty}R_{\p} +x_iR_{\p}}\right) \cdot e(R/\p)\, ,
\end{equation*}
where we use the convention that $(x_1, \ldots, \, x_{i-1}):\, I^{\infty}$ is $\, 0 \, $ for $i=0$ and is $\, 0:I^{\infty}$ for $i=1$. 
\end{proposition}

Notice that the case $i=0$ is Formula (\ref{0}).

\begin{proof} 
Write $\m$ for the maximal ideal of $R$. We may assume that $I\neq R$ and that $i\le d$ since otherwise both sides of the equation are zero. Achilles and Manaresi proved the above formula for a sequence $x_1, \ldots, x_i$ such that the images of $x_1, \ldots, x_i$  in $I/ \m I$ avoid a finite number of proper subspaces of $I/ \m I$. 
This implies that the formula holds for general elements $x_1, \ldots, x_i$  of $I$. 
\end{proof}

The above length formula can be used to derive the following properties of the multiplicity sequence.  

\begin{corollary}\label{COR} 
Let $(R,\m) $ be a Noetherian local ring of dimension $d$ and $I$ an ideal.  
\begin{enumerate}[$(a)$]
\item\label{modulo} If $H$  is an ideal contained in $0:I^{\infty}$ and $\, \dim \, R/H= d$, then for $i\ge 1$ 
$$c_i(I,R)=c_i(I, R/H)\, ;$$
\item\label{grade}  If $k$ is infinite, ${\rm grade}\ I \ge 1$, and $x$ is a general element of $I$, then 
$$c_1(I,R) = c_0(I,R/(x)) \, ;$$
\item\label{section} If $k$ is infinite, ${\rm ht}\ I \ge 1$, and $x$ is a general element of $I$, then for $i\ge 2$
$$c_i(I,R)=c_{i-1}(I, R/(x))\, ;$$
\item\label{variable} If $S = R[y]_{(\m,y)}$, where $y$ is an indeterminate, then $\, c_0((I,y),S) =0 \,$ and for $i \ge 1$
$$c_i((I,y),S) = c_{i-1}(I,R).$$
\end{enumerate}
\end{corollary}

\begin{proof} 
To prove item (\ref{modulo}) let $z$ be an indeterminate over $R$. Replacing $R$ by $R(z):=R[z]_{\m[z]}$
and $I$ by $IR(z)$ does not change $c_i(I,R)$ or $c_i(I, R/H)$. Thus we may assume that $k$ is infinite. We
use the notation of Proposition~\ref{Lformula}.
Notice that for $i \geq 1$,
$(x_1, \ldots, x_{i-1},H):I^{\infty}=(x_1, \ldots, x_{i-1}):I^{\infty}$ because $$(x_1, \ldots, x_{i-1},H):I^{\infty} \subset (x_1, \ldots, x_{i-1}, 0:I^{\infty}):I^{\infty} 
= (x_1, \ldots, x_{i-1}):I^{\infty}.$$ Now the assertion is a
direct consequence of the Length Formula of Proposition~\ref{Lformula}.  

For the proof of item (\ref{grade}) we notice that $I$ is not contained in any associated prime ideal of $R$ since ${\rm grade}\ I \ge 1$. Therefore, $0: I^\infty = 0$. 
Moreover, ${\rm dim }(R/xR) = d-1$ because $x$ is not contained in any associated prime ideal of $R$. 
Applying Proposition \ref{Lformula} with $i=1$ and $x_1 = x$ and with $i=0$, respectively,  we obtain
$$c_1(I,R) = \sum_{\p \in V(I),\ \dim R/\p = d-1}\lambda(R_\p/xR_\p) \cdot e(R/\p) = c_0(I,R/xR).$$

Item (\ref{section})  follows from Proposition~\ref{Lformula}, with $x_1=x$.  Indeed, 
${\rm dim} (R/xR) = {\rm dim} \, R-1$ because $I$ is not contained in any minimal prime ideal of $R$.

Item (\ref{variable}) follows, most directly, from the definition of the 
multiplicity sequence. Indeed,
$$G_{(\m,y)}(G_{(I,y)}(S))=G_{\m}(G_{(I,y)}(S))=G_{\m}(G_{I}(R)[y^{\star}])=G_{\m}(G_{I}(R))[y^{\star}] 
\, ,$$
where $y^{\star}$ is a variable of degree $(0,1)$.
Now, comparing the bigraded Hilbert functions of $G_{\m}(G_{I}(R))[y^{\star}]$ 
and $G_{\m}(G_{I}(R))$
yields the result.
\end{proof}

A Noetherian ring is called {\em equidimensional} if every minimal prime has the same dimension and {\em catenary} if any two maximal strictly increasing chains of prime ideals between two given prime ideals $\p_1 \subset \p_2$ have the same length. 

\begin{remark} \label{htt}
If $R$ is an equidimensional and catenary Noetherian local ring and $\mathfrak a$ is an ideal, then $\dim R - \dim R/{\mathfrak a} = {\rm ht} \, \mathfrak{a}.$
Therefore, we may replace the condition $\, \dim \, R/\p = d-i \, $ by $\, {\rm ht} \, \/ \p = i \,$ in Proposition~\ref{Lformula} (and Formula (1)), and obtain
\begin{equation*}
c_i(I,R)=\sum_{\substack{\p\in V(I), \ {\rm ht}\, \p = i\\
\p \supset (x_1, \ldots, \, x_{i-1}):\, I^{\infty}}
}  \lambda \left(\frac{R_{\p}}{ (x_1, \ldots, x_{i-1})R_{\p}:I^{\infty}R_{\p} +x_iR_{\p}}\right) \cdot e(R/\p) \quad \ \mbox{ for  } \ i \geq 0\, .
\end{equation*}
\end{remark}

\medskip

Multiplicity based criteria usually require the ambient local ring $R$ to be equidimensional and {\em universally catenary}, meaning that all finitely generated $R$-algebras are
catenary. This assumption is used for instance in Rees' multiplicity theorem for zero-dimensional ideals $I \subset J$, which says that $J$ is integral over $I$ if (and only
if) $e(I,R)=e(J,R)$ \cite{R61}.  Owing to Ratliff (see e.g. \cite[31.6 and 31.7]{Mat}), a Noetherian local ring $R$ is equidimensional and universally catenary if and only if 
$R$ is formally equidimensional or quasi-unmixed, meaning the completion $\widehat{R}$ is equidimensional. 
We now collect additional properties of the multiplicity sequence in this slightly more restrictive setting.

\vspace{0.05in}

\begin{proposition}\label{COR3} 
Let $R $ be an equidimensional and  catenary Noetherian local ring  
and $I$ an ideal. \begin{enumerate}[$( a )$]
\item\label{g} If $\, i \le{\rm ht} \ I$, 
then  
\begin{equation*} 
c_i(I)=\sum_{\p\in V(I),\
 { \rm ht}\, \p = i}
 e(I,R_{\p}) \cdot e(R/\p)\,;
\end{equation*}
 \item\label{ht}  ${\rm ht} \ I= {\rm min}\{i \, | \, c_i(I)\not=0\} ;$ 
\item\label{ell} If $R$ is universally catenary and $I \neq R$, then $\ell(I)={\rm max}\{i \, | \, c_i(I)\not=0\} .$ 
\end{enumerate}
\end{proposition}
\begin{proof} Recall that $c_i(I)=0$ if $i < {\rm ht} \ I$ or $i> \ell(I)$. Now
item (a) follows from \cite[2.3(iii)]{AM97}, and
(b) is an immediate consequence of (a). Part (\ref{ell}) follows from
\cite[2.3(ii)]{AM97} and 
the fact that the associated graded ring $G_I(R)$ is equidimensional and catenary, see \cite[proof of 3.8]{R1}. 
 \end{proof}


Now we want to compare the multiplicity sequence of an ideal with that of its localizations.  
The following lemmas allow us to work with elements that are general in an ideal and in its localizations at finitely many primes. 

\begin{lemma}\label{open}
\begin{enumerate}[$($a$)$]
\item Let $\kappa \subset K$ be a field extension with $\kappa$ infinite.
Every dense open subset of $K^n$ contains a dense open subset of $\kappa^n$.
\item Let $A$ be a discrete valuation ring with infinite residue field $\kappa$ and quotient field $Q$, and 
consider the natural maps $\pi: A^n \twoheadrightarrow \kappa^n$ and $\eta: A^n \hookrightarrow Q^n$.
For every dense open subset $U$ of  $Q^n$ there exists a dense open subset
$W$ of $\, \kappa^n$ such that  $\, \eta(\pi ^{-1}(W)) \subset U$. 
\end{enumerate}
\end{lemma}

\begin{proof} (a) Since any field extension is a purely transcendental extension followed by an algebraic
extension, we may assume that the field extension $\kappa \subset K$ is either algebraic or purely transcendental. 

Let $U$ be a dense open subset of $K^n$. We may suppose that $U$ is a basic open subset, say $U= K^n \setminus V(f)$ with $0 \neq f \in K[x_1, \ldots, x_n]$.
We need to prove that $V(f) \cap \kappa^n \subset V(I)$ for some nonzero ideal $I$ of
the polynomial ring $\kappa[x_1, \ldots, x_n]$. This is clear if the extension $\kappa \subset K$ is algebraic  because then $\kappa[x_1, \ldots, x_n] \subset K[x_1, \ldots, x_n]$ is an integral extension of domains and therefore 
the ideal generated by $f$ contracts to a nonzero ideal $I$ of $\kappa[x_1, \ldots, x_n]$. If the extension $\kappa \subset K$ is purely transcendental, then $K$ is the 
quotient field of a polynomial ring $\kappa[\{y_i\}]$. After clearing denominators,
we may assume that $f \in \kappa[\{y_i\}, x_1, \ldots, x_n]$. We think of $f$ as a polynomial in the variables $y_i$ and
let $I \subset \kappa[x_1, \ldots, x_n]$ be the ideal generated by its coefficients. Then $I \neq 0$ since $f \neq 0$,
and $V( f) \cap \kappa^n =V(I)$ since the elements $y_i$ are algebraically independent over $\kappa$.
\smallskip

\noindent
(b) Again we may assume that $U$ is a basic open set, say $U= Q^n \setminus V(f)$ with $0 \neq f \in Q[x_1, \ldots, x_n]$.
Multiplying $f$ by a power of the uniformizing parameter $t$ of $A$, with exponent in $\mathbb Z$, we 
may assume that $f \in A[x_1, \ldots, x_n] \setminus  (t A)[x_1, \ldots, x_n]$. Let $\overline f$ be the
image of $f$ in $\kappa[x_1, \ldots, x_n]$ and notice that $\overline f \neq 0$. The dense open subset
$W:=D_{\overline f} \, $ has the desired property.
\end{proof} 

\vspace{0.0in}

\begin{lemma}\label{general} 
Let $(R, \m, k)$ be a local ring and assume that  $k$ is not an algebraic extension of a finite field. Let $R_0$ be the prime ring of $R$ and $k_0$ the prime field of $k$.
If $\, {\rm char} (k_0) =0 \,$ let $y$ be an element of $R_0 \,$, and if  $\, {\rm char} (k_0) > 0 \,$
let $y$ be a preimage in $R$ of an element of $k$ that is algebraically independent over $k_0$ $($such an element exists by our assumption$)$. 
Set $A:=(R_0[y])_{\m \cap R_0[y]}$ and let $\kappa:=k_0(y)$ be the residue field of $A$.  
Let $\{\p_1, \ldots, \p_s\}$ be a finite subset of $\, {\rm Spec}(R)$ and let $U_i$ be dense open subsets of $\, k(\p_i)^n$. 
There exists a dense open subset $U$ of $\, {\kappa }^n$ such that 
whenever the image of $\, (\lambda _1, \ldots, \lambda _n)\in A^n$ in $\kappa^n$ belongs to $U$ then the image of $\, (\lambda _1, \ldots, \lambda _n)$  in $k(\p_i)^n$ 
belongs to $U_i \,$ for every $1\le i \le s$. 
\end{lemma}

\begin{proof} 
Write $p:={\rm char} (k) \geq 0$. Notice that $A$ is a local principal ideal ring with infinite residue field $\kappa$ and
maximal ideal $\n:=pA$. If $A$ is not Artinian, then $A$ is a discrete valuation ring with $Q:={\rm Quot}(A)$.  
  
The prime ideals $\{\p_1, \ldots, \p_s\}$
contract to $\n$ if $A$ is Artinian, and to $\n$ or $0$ if $A$ is a discrete valuation ring;
the corresponding residue field extensions are $\kappa \subset k(\p_i)$, and
$\kappa \subset k(\p_i)$ or $Q \subset k(\p_i)$, respectively. 
It suffices to show our assertion for one $\p_i$. If $\kappa \subset k(\p_i)$, we apply Lemma \ref{open}(a) to this field extension. 
If on the other hand $Q \subset k(\p_i)$, we apply Lemma \ref{open}(a) to this field extension and then Lemma \ref{open}(b). 
\end{proof}

Recall that a local ring is called {\em analytically unramified} if its completion is a reduced ring.

\begin{proposition}\label{COR2} 
Let $R $ be an equidimensional and universally catenary Noetherian local ring and $I$ an ideal. 
Let $\p$ be a prime ideal and assume that $R/\p$ is analytically unramified. Then for $i\ge 0$
$$c_i(I, R_{\p})\le c_i(I, R)\, .$$
\end{proposition}

\begin{proof}  
The localization $R_{\p}$ is also equidimensional and universally catenary. After a purely trans-cendental residue field extension as in the proof of
Corollary~\ref{COR}, we may 
assume that the residue field of $R$ is not an algebraic extension of a finite field.
Write $\m$ for the maximal ideal of $R$. 
An application of Lemma~\ref{general}, with $\p_1:= \m$ and $\p_2:=\p$, shows that we can use the same elements $x_1, \ldots, x_i$ 
in the length formula of Remark~\ref{htt} to compute $c_i(I, R)$ and $c_i(I, R_{\p})$.
Now, to deduce our assertion we only need to prove that $e((R/\q)_{\p}) \leq e(R/\q)$ whenever $\q\subset \p$.
By \cite[40.1]{N} this inequality holds because $\, \dim \, R/\p + {\rm ht} \, (\p/\q) = \dim \, R/\q \, $ and $R/\p$ is analytically unramified. 
\end{proof}

\smallskip

\section{The key technical result}

The aim of this section is to prove a technical result, Theorem~\ref{MT}. This result will play a crucial role in our solution to Conjecture \ref{Q}.
We begin with a lemma establishing an inequality between Hilbert-Samuel multiplicities.

\begin{lemma}\label{SL} 
Let $R $ be a Noetherian local ring.  Let $\, \underline{x}, \underline{t} \,$ be a system of parameters and assume that $\underline{x}$ is a regular sequence on $R$. Then $$e((\underline{x}), R/(\underline{t})) \ge e((\underline{t}), R/(\underline{x})) \, .$$
\end{lemma} 
\begin{proof} 
Since $\underline{t}$ is part of a system of parameters, we have
$$e((\underline{x}), R/(\underline{t})) \geq e((\underline{x}, \underline{t}), R)$$
by  \cite[1.2.12]{FOV}, and as 
$\, \underline{x}, \underline{t} \,$ is a system of parameters and $\underline{x}$ is a regular sequence, 
$$e((\underline{x}, \underline{t}), R) = e((\underline{t}), R/(\underline{x}))$$ 
according to \cite[1.2.14]{FOV}. Alternatively, one can use  the multiplicity formula of Auslander and Buchsbaum for systems of parameters  \cite[4.3]{AB}.
\end{proof}

By $\Min(\cdot)$ we denote the set of minimal prime ideals of a given ideal or of the ideal generated by a given collection of elements.

\begin{lemma} \label{sop}
Let $R$ be a Noetherian local ring and let $\underline{t}$ be part of a system of parameters of $R$. Then
$$\sum_{\p \in \Min(\underline t)} e((\underline{t}), R_\p) \cdot e(R/\p) \ge  \ e(R)\, .$$
\end{lemma}

\begin{proof}
We may assume that the residue field  of $R$ is infinite.
Let $\underline{y}$ be a sequence of general elements of the maximal ideal of $R$ that form a system of parameters of $R/(\underline{t}).$ 
Notice that the image of $\underline{y}$ in $R/(\underline t)$ generates a minimal reduction of the maximal ideal, and hence the
image of $\underline y$ in $R/\p$ generates a reduction of the maximal ideal  for every $\p \in \Min(\underline t)$. From this it follows that 
$e(R/\p) = e((\underline{y}), R/\p)$. Therefore,
$$\sum_{\p \in \Min(\underline t)} e((\underline{t}), R_\p)\cdot e(R/\p) = \sum_{\p \in \Min(\underline t)} e((\underline{t}), R_{\p}) \cdot e(\underline{y}, R/\p) \ge e((\underline{t}, \underline{y}), R) \geq e(R) \, ,$$
where the first inequality holds by  \cite[24.7]{N}.
\end{proof}
 
The next theorem provides a condition, in terms of multiplicity sequences, for when 
a collection of elements is transversal to every prime ideal $\p \in L(I)$. For an ideal $I$ of a Noetherian ring $R$, we denote by $L(I)$
the set of prime ideals $\p \in V(I)$ where the ideal $I$ has maximal analytic spread, namely $\ell(I_\p) = {\rm ht}\, \p$.
The set $L(I)$ is finite. Indeed, if $\p \in L(I)$ then $\p$ is the contraction of a minimal prime of the associated graded 
ring $G_I(R)$ (see also  \cite[3.9 and 4.1]{Mc}). The converse holds whenever $R$ is equidimensional, universally catenary, and local.
As we will see in Theorem~\ref{Mc}, $L(I)$ is also the collection of prime ideals that are critical for proving integral dependence.

\begin{theorem}\label{MT} 
Let $R$ be an equidimensional and universally catenary Noetherian local ring of dimension $d$. Let
$\underline{t}=t_1, \ldots, t_r$ be elements in $R$ that form part of a system of parameters of $R$. Let $I$ be an ideal of $R$ and
assume that $\,{\rm ht} \, (\underline{t}, I, 0:I^{\infty}) > r$. If
$$c_i(I, R/(\underline{t})) \leq c_i(I, R) \quad \mbox{ for  } \ 1\le i \le d-r\, ,$$ 
then   $t_1, \ldots, t_r$ form part of a system of parameters of $R/\p\, $ for every  $\p\in L(I)$. 
\end{theorem}

\begin{proof} 
We may assume that $I$ is not nilpotent.
Otherwise $I_{\p}$ is nilpotent for every $\p \in L(I)$, hence
$\ell(I_{\p})=0 \,$ and so $\, {\rm ht} \, \p =0 \,$ because
$\ell(I_{\p})= {\rm ht} \, \p$.
In this case $t_1, \ldots, t_r$  are  part of a system of parameters of $R/\p$ because $R$ is equidimensional. We may also
suppose that $I \neq R$ since otherwise $L(I) =\emptyset$. Thus the ideal $(\underline{t}, I, 0:I^{\infty})$ is proper and
therefore has height at most $d$. It follows that $r <d$.

We argue that we may replace $R$ by $\overline{R} =R/(0:I^{\infty})$. 
First, we show our assumptions are preserved. Since $I$ is not nilpotent, it follows that ${\rm ht}\, (0:I^{\infty})=0$. Moreover, every associated prime
of the ideal $0: I^{\infty}$ is an associated prime of $\, 0$. Hence, as $R$ is equidimensional,
$\overline{R}$ is equidimensional of dimension $d$ and $\underline{t}$ is part of a system of parameters of  $\overline{R}$.
Since $R$ and $\overline R$ are equidimensional and catenary, we also have ${\rm ht}\, (\underline{t}, I) \overline{R}>r$. 
For $i \ge 1$, the numbers $c_i$ do not change upon factoring out $0:I^{\infty}$, as can be seen from Corollary~\ref{COR}(\ref{modulo}).

Next, we show that if the conclusion of the theorem holds for $I\overline{R}$, it also holds for $I$. 
Let $\p\in L(I)$. 
If $\p \not\in V(0:I^{\infty} )$, then $I_{\p}$ is nilpotent. This implies $\ell(I_\p) = 0$ and hence
${\rm ht} \, \p =0$.  Now as before, 
$t_1, \ldots, t_r$ are part of a system of parameters of $R/\p$.
If $\p \in  V(0:I^{\infty} )$, then $\p \, \overline R \in {\rm Spec}(\overline R)$ and, as before, ${\rm ht}\, \p \, \overline{R}= {\rm ht} \, \p$. By the Artin-Rees Lemma
we have
$I^n \cap (0:I^{\infty} )=0$ for $n\gg 0$, hence $(I\overline R)^n=I^n$ for $n \gg 0$. Therefore, $\ell(I \overline{R}_{\p}) =\ell(I_{\p})= {\rm ht}\, \p = {\rm ht} \, \p \, \overline{R}$.
Since the assertion of the theorem holds for $I\overline{R}$, we conclude that $t_1, \ldots, t_r$ are part of a system of parameters of $\, \overline R/\p \, \overline R=R/\p$.

As $R$ can be replaced by $\overline{R}$, we may assume that $0:I^\infty = 0$ or, equivalently, ${\rm grade} \, I \ge 1$. 
Now, we are going to prove the theorem by induction on $\ell := \ell(I)$. 

Let $\ell \leq 1$. If $\p\in L(I)$,
then ${\rm ht} \, \p = \ell(I_\p) \le \ell \leq 1$ and therefore
$$ {\rm ht} \, (\underline{t}, \p) \geq {\rm ht} \, (\underline{t}, I) \geq r+1 \geq r + {\rm ht} \, \p \, .$$
Thus $t_1, \ldots, t_r$ form part of a system of parameters of $R/\p$, again since $R$ is equidimensional 
and catenary.

Let $\ell \geq 2$. After a purely transcendental residue field extension, we may assume that the 
residue field of $R$ is not algebraic over a finite field. 
Let $x$ be a general $A$-linear combination of a finite
generating set of $I$ as in Lemma \ref{general}, and keep in mind that by the same lemma, $x$ is a general element of $I$. 
Since
${\rm grade} \, I \ge 1$, it follows that $x$ is a non zerodivisor on $R$. As moreover $\ell \ge 1$, the element $x$ is part of a minimal generating set of a minimal reduction of $I$. Thus  $\ell(I S) \le \ell -1$, where $S:=R/(x)$.

We show that our assumptions pass from $I \subset R$ to $I S \subset S$. Clearly $S$ is equidimensional and universally catenary.
Since ${\rm ht} \,(\underline{t}, I) \geq r+1$
and $x$ is a general element  of $I$, we also have ${\rm ht} \, (\underline{t}, x) \geq r+1$. Therefore $t_1, \ldots, t_r$ 
form part of a system of parameters of $S$. 
Notice that $\, {\rm ht} \, I(R/(\underline{t})) \ge {\rm ht}\,  (\underline{t}, I) - r \ge 1$ and that the image of $x$ is a general element of $I(R/(\underline{t}))$. 
By Corollary \ref{COR}(\ref{section}), we have for $i \ge 2$, $c_i(I, R/(\underline{t}))=c_{i-1}(I, S/(\underline{t}))$ and $c_i(I, R)=c_{i-1}(I, S)$.
Therefore,
$$c_i(I, S/(\underline{t})) \leq c_i(I, S) \quad \mbox{ for  } \ 1\le i \le (d-1)-r=\dim S-r\, ,$$ 
as required.

It remains to prove that $\,{\rm ht}\, (\underline{t}, I, (x):I^{\infty})S> r$ or, equivalently, that
$$\,{\rm ht}\, (\underline{t}, I, (x):I^{\infty})> r+1 \, .$$
Since the ideal $(\underline{t}, I)$ has height at least $r+1$ by assumption, there are at most finitely many
prime ideals of height $r+1$ that contain it. Let $\Lambda$ be the set of these prime ideals,
$$\Lambda = \{\p \in V(\underline{t}, I)|\ {\rm ht}\ \p = r+1\}\, .$$
If $\Lambda=\emptyset$, then $ {\rm ht}\, (\underline{t}, I ) >r+1$ and we are done. Otherwise, we need to show that for every $\p \in \Lambda$ one has 
$(x) : I^{\infty} \not\subset \p$ or, equivalently, $I_{\p} \subset \sqrt{(x)_{\p}}$. To this end, fix $\p \in \Lambda$ and let $\Sigma_{\p}$ be the set of all minimal prime ideals of $(x)$ that are contained in $\p$, 
$$\Sigma_{\p}= \{\q \in {\rm Min} (x)|\ \q\subset \p\}\, .$$
Notice that these prime ideals have height one.
 Let $\Gamma_{\p}$ be the set of all prime ideals of height one that contain $I$ and are contained in $\p$,
 $$\Gamma_{\p}= \{\q \in V(I)|\ \q\subset \p  \mbox{\ and \ }   {\rm ht}\ \q = 1\}\, .$$ 
To prove that $I_{\p} \subset \sqrt{(x)_{\p}}\ $ it suffices to show that the inclusion $\Gamma_{\p}\subset \Sigma_{\p}$ is an equality. 

Finally, we introduce the set $\Gamma$ of all prime ideals of height one that contain $I$,
$$ \Gamma  = \{\q \in V(I)|\  {\rm ht}\ \q = 1\}.
$$
Since $\Gamma \subset \Min(I)$ because ${\rm ht}\, I \ge 1$, the set  $\Gamma$ is finite as well. Moreover, for all $\q \in \Gamma$ we have
\begin{equation}\label{setmin}\{\p \in \Lambda|\ \p \supset \q\} = \Min((\underline t)(R/\q))\end{equation}
because the minimal prime ideals of $(\underline{t}, \q)$ have height $r+1$. 
By Lemma \ref{general}, the image of $x$ is a general element of the ideals $\, I \,(R/(\underline{t}))_{\p}$ 
for each of the finitely many $\, \p\in \Lambda$. 
In particular, $x$ generates a reduction of these ideals, as they are ideals of one-dimensional rings. Also recall that $x$ is a non zerodivisor on $R$.  

To prove that $\Sigma_{\p}=\Gamma_{\p}$, we compare $c_1(I, R/(\underline{t}))$ and $c_1(I, R) $. We have
\vspace{.003in}
$$
\begin{array}{llll} 
\!\!\! c_1(I, R/(\underline{t}))&=&\displaystyle\sum_{ \p \in \Lambda} e(I, (R/(\underline{t}))_{\p}) \cdot e(R/\p) & \mbox{by Proposition~\ref{COR3}(\ref{g})} \\
&=& \displaystyle\sum_{ \p \in \Lambda} e((x), (R/(\underline{t}))_{\p}) \cdot e(R/\p)& \mbox{since $x$ generates a reduction of $I\, (R/(\underline{t}))_{\p}$} \\
&\ge & \displaystyle\sum_{ \p \in \Lambda} e((\underline{t}), (R/(x))_{\p}) \cdot e(R/\p) &     \mbox{by Lemma~\ref{SL} since $x$ is regular} \\
&=&\displaystyle\sum_{ \p \in \Lambda} \displaystyle\sum_{\q\in \Sigma_{\p}}  \lambda((R/(x))_{\q}) \cdot  e((\underline{t}),(R/\q)_{\p}) \cdot  e(R/\p) &  \mbox{by the associativity formula} \\
&\ge & \displaystyle\sum_{ \p \in \Lambda} \displaystyle\sum_{\q\in \Gamma_{\p}}  \lambda((R/(x))_{\q}) \cdot  e((\underline{t}),(R/\q)_{\p}) \cdot  e(R/\p)&  \mbox{since $\Gamma_{\p}\subset \Sigma_{\p}$}\\
&\ge & \displaystyle\sum_{\q\in \Gamma}  \lambda((R/(x))_{\q})  \sum_{\p \in \Lambda,\ \p\supset \q} e((\underline{t}),(R/\q)_{\p}) \cdot  e(R/\p) & \mbox{by switching the summation} \\ 
&\ge & \displaystyle\sum_{\q\in \Gamma}  \lambda((R/(x))_{\q}) \cdot e(R/\q) &  \mbox{by  Lemma \ref{sop} and (\ref{setmin})}\\
&\ge & \displaystyle\sum_{\q\in \Gamma}  e((x), R_{\q}) \cdot e(R/\q ) &  \mbox{by \cite[14.10]{Mat}}\\
&\ge & \displaystyle\sum_{\q\in \Gamma}  e(I, R_{\q}) \cdot e(R/\q ) &  \mbox{since $x \in I$}\\
&=  & c_1(I, R) & \mbox{by Proposition~\ref{COR3}(\ref{g})} \\
&\ge & c_1(I, R/(\underline{t})) & \mbox{by assumption as $r<d$}. 
\end{array}
$$
\smallskip
\noindent
It follows that all inequalities above are equalities. In particular, $\Sigma_{\p}=\Gamma_{\p}$ for every $\p\in \Lambda$, as asserted. 

We have now shown that our assumptions pass from $I \subset R$ to $I S \subset S$. Since $\ell(I S)\le \ell -1$, the induction hypothesis shows that  the assertion of the theorem holds for $I S \subset S$. 
To lift the assertion from $I S$ back to $I$, recall that $x$ is a non zerodivisor on $R$. Fix $\p \in L(I)$. 
By Lemma \ref{general}, the element $x$ is general in $I_{\p}$ and hence superficial. It follows that the preimage of any reduction of $I S_{\p}$ is a reduction of $I_{\p}$, see for instance \cite[8.6.1]{SH}, which gives 
$\ell(I S_{\p})\ge \ell(I_{\p})-1 = {\rm ht}\, \p -1={\rm ht}\, \p S,$
and hence $\ell(I S_{\p})={\rm ht}\, \p S$. Thus, by the induction hypothesis $t_1, \ldots, t_r$ form part of a system of parameters of $S/\p S=R/\p$, as required. 
\end{proof}

\smallskip

\section{Integral Dependence}

We begin by recalling the known fact that integral dependence over an ideal $I$ can be checked locally at the finitely many prime ideals in $L(I)=\{ \p \in V(I) \mid  \ell(I_\p) = {\rm ht}\, \p\}$.  

\begin{theorem} \label{Mc}
Let $R$ be an equidimensional and universally catenary Noetherian local ring and let $I \subset J$ be ideals. The ideal $J$ is integral over $I$ if and only if $J_\p$ is integral 
over $I_\p \, $ for every prime ideal $\p \in L(I)$.
\end{theorem}

\begin{proof} Let $\overline I$ denote the integral closure of $I$. It suffices to prove that  
$J \subset \overline I$ if $J_{\p} \subset ({\overline I})_{\p} \, $ for every $\p \in L(I)$. This follows because every associated prime of $\, \overline I$ belongs to $L(I)$ by \cite[3.9 and 4.1]{Mc}.
\end{proof}

We will use Theorem \ref{Mc} to prove Conjecture \ref{Q}.
The main idea is to replace the ideals $I \subset J$ by ideals $I^* \subset J^*$ in a new local ring $S$ which contains an element $t$ such that 
$I^*_\p = J^*_\p$ if $t \not\in \p$ and to use Theorem \ref{MT} to show that the condition $c_i(I, R)  = c_i(J,R)$ forces $t \not\in \p$ for every $\p \in L(I^*)$. Then $J^*$ is integral over $I^*$ by Theorem \ref{Mc}, which implies that $J$ is integral over $I$.

\smallskip
  
\begin{theorem}[\bf Integral Dependence]\label{ID}  
Let $R$ be an equidimensional and universally catenary Noetherian local ring of dimension $d$ and let $I \subset J$ be ideals.
The following are equivalent$\,:$
\begin{enumerate}[$(1)$]
\item $c_i(I)\le c_i(J) \quad \mbox{for \ \ } 0\le i \le d\, ;$
\item $c_i(I)= c_i(J) \quad \mbox{for \, } 0\le i \le d \, ;$
\item $J$ is integral over $I .$ 
\end{enumerate}
\end{theorem}

\begin{proof} 
That (3) implies (2) was proved in \cite[2.7]{C} (see also \cite[11.5]{TV}). Since (2) implies (1), 
we only need to show that (1) implies (3). Write $\m$ for the maximal ideal of $R$.  Replacing $R$ by the localized polynomial ring $R[y]_{(\m,y)}$ and $I$, $J$ by the ideals $(I,y)$, $(J,y)$, we may suppose that ${\rm ht}\ I >0$. According to Corollary~\ref{COR}(\ref{variable}), the inequalities in $(1)$ are preserved. 

Consider the localized polynomial ring $S = R[t]_{(\m,t)}$ and the ideal $H = I S +tJ S \subset J S$.  
One has ${\rm ht}\, (t,H, 0:H^{\infty})\ge {\rm ht}\, (t, I)>1$. 
Notice that 
$\, J \, S_{\m S} = H \,S_{\m S}\,$ because 
$t$ is a unit in $S_{\m S}$. 
For $1\le i \le {\rm dim }\, S-1=d \,$ we obtain
$$c_i(H, S/(t))=c_i(I, R)\le c_i(J, R)=c_i(J, S_{\m S}) = c_i(H, S_{\m S})  \le c_i(H, S) \, ,$$
where the last inequality follows from Proposition~\ref{COR2} since $S/\m S$ is analytically unramified.
By Theorem~\ref{MT}, $t \not\in \p$ for every prime ideal $\p \in L(H)$. 
Therefore, $H_\p = (J S)_\p$ for all such primes $\p$. 
By Theorem~\ref{Mc}, this implies that $J S$ is integral over $H$. Reducing modulo $t$ we see that $J$ is integral over $I$. 
\end{proof}

\begin{remark} 
The idea of considering the ideal $H = I S+tJ S$ in the localized polynomial ring $S = R[t]_{(\m,t)}$ is due to Gaffney and Gassler \cite[proof of 4.9]{GG}. In the analytic set-up, $H$ is a family of ideals parametrized by $t$ with 
$H(0) = I$ and $H(t) = J$ for $t \neq 0$. The assumption $c_i(I)= c_i(J)$ for $0 \le i \le d$ means that the map 
$t \mapsto \left( c_0(H(t)),...,c_d(H(t)) \right)$ is constant. By the principle of specialization of integral dependence proved by Gaffney and Gassler \cite[4.7]{GG}, this implies that $J$ is integral over $I$. Their proof of the principle of specialization of integral dependence in the analytic case is intricate.
\end{remark}

Our approach can  also be used to prove the following principle of specialization of  integral dependence (PSID) for arbitrary ideals.

\begin{theorem}[\bf PSID]\label{PSID}\ 
Let $\varphi: T \to R $ be a local homomorphism of Noetherian local rings. Assume that
$T$ is regular with residue field $k$ and quotient field $L$, that $R$ is equidimensional and universally catenary, 
and that $\dim \, k \otimes_TR = \dim R - \dim T$. Further suppose that there is a homomorphism of rings $\, \psi: R \to T$ with $\, \psi \varphi={\rm id}$
and write $\wp = {\rm ker} \, \psi$. 
Let $I$ be an ideal of $R$ such that $\, {\rm ht} \  I \, (k\otimes_TR) >0$ and let $J \supset I$ be another ideal.

If $\, L\otimes_TJ$ is integral over $L\otimes_TI$ and
$$c_i(I,k\otimes_TR) \leq c_i(I,R_\wp) \quad \mbox{ for  } \ 1\le i \le \dim \, k\otimes_TR\, ,$$ 
then $J$ is integral over $I$.
\end{theorem}

Notice that $R_\wp = (L\otimes_TR)_{L\otimes_T\wp}$, where $L \otimes_T\wp$ is a prime ideal of $L \otimes_TR$, because $\varphi^{-1}(\wp)=0$. 
Here $L \otimes_TR$ is the ring of the generic fiber of $\varphi$. Thus the PSID above says, in particular, that $J$ is integral over $I$ on the total space of the family, if it is so on the
generic fiber and if the multiplicity sequence of $I$ on the special fiber coincides with the one on the generic fiber locally along the parameter space $V(\wp)$. 

\begin{proof}  
By Theorem \ref{Mc}, it suffices to show that $J_{\p}$ is integral over $I_{\p}$ for every $\p\in L(I)$. 
To do so, we apply Theorem~\ref{MT}, with ${\underline t} = t_1, \ldots, t_r$  the image in $R$ of a regular system of parameters of $T$. 

Notice that $R/({\underline t}) = k\otimes_TR$. Hence by our hypotheses, $\underline t$ form part of a system of parameters of $R$ and
$\,{\rm ht} \, (\underline{t}, I) > r$. Moreover,
$c_i(I, R_\wp) \le c_i(I, R)$ by Proposition~\ref{COR2}
since $R/\wp \cong T$ is analytically unramified.  Thus,
$$c_i(I, R/({\underline t})) \le c_i(I, R)$$
for $1\le i \le \dim \, k\otimes_TR = \dim R -r$. 

Let $\n$ denote the maximal ideal of $T$. Theorem~\ref{MT} implies that $\, {\rm ht}\, \n=r = {\rm ht}\, \n  (R/\p)$
for every $\p\in L(I)$. On the other hand, ${\rm ht}\, \n  (R/\p)\le {\rm ht}\, \n/ \varphi^{-1}(\p)$ since by Krull's Altitude Theorem, the height of the maximal ideal of a Noetherian local ring cannot increase
when extended to a Noetherian extension ring. Thus, $\, {\rm ht}\, \n    \le {\rm ht}\, \n/ \varphi^{-1}(\p)$. We deduce that $\varphi^{-1}(\p)=0$ as $T$ is a domain. In other words, $R_{\p}$ is  a localization of $L\otimes_TR$, and so $J_{\p}$ is integral over $I_{\p}$ by assumption. 
\end{proof}

\vspace{0.0in}

\begin{remark}
This proof shows that the conclusion of Theorem \ref{PSID} holds with the weaker, though geometrically less significant, hypothesis that $\, c_i(I,k\otimes_TR) \leq c_i(I,R)$ for $1\le i \le \dim \, k\otimes_TR$.  
\end{remark}

\smallskip

\section{Multiplicity sequence and local $j$-multiplicities}

In this section we discuss the relationship between the multiplicity sequence and the $j$-multiplicity of an ideal with respect to integral dependence.

Let $(R,\m)$ be a Noetherian local ring of dimension $d$ and $I$ an ideal in $R$.
Let $G := \oplus_{n \ge 0}\, I^n/I^{n+1}$ be the associated graded ring of $I$. 
The {\em $j$-multiplicity} of $I$ was introduced by Achilles and Manaresi \cite{AM} as the invariant
$$j(I) := \sum_{\p \in V(\m G),\ \dim G/\p = d} \lambda(G_{\p})\cdot e(G/\p) \, .$$
It can be also interpreted as the multiplicity of the graded module $H_\m^0(G)$ \cite[Section 6.1]{FOV}.

Note that there exists $\p \in V(\m G)$ with $\dim G/\p = d \,$ if and only if $\dim G/\m G = d$.
Since $F(I) = G/\m G$ for $I\not=R$ and $\ell(I) = \dim F(I)$, it follows that $j(I) \neq 0$ if and only if $\ell(I) = d$ and $I\not=R$.
Thus, the $j$-multiplicity of $I$ is supported precisely on the set $L(I)$, meaning that $L(I)=\{  \p\in{\rm Spec}(R) \mid   j(I_{\p})\neq 0\}$. 

The $j$-multiplicity can be considered as a generalized Hilbert-Samuel multiplicity, because $j(I) = e(I,R)$ when $I$ is an $\m$-primary ideal. In general, we have $j(I) = c_d(I)$  \cite[2.4(ii) and 2.3(i)]{AM97}.
 
It follows from the work of Flenner and Manaresi \cite[3.3]{FM} that two arbitrary ideals $I \subset J$ in an equidimensional and universally catenary Noetherian local ring have the same integral closure if and only if $j(I_\p) = j(J_\p)$ for all $\p \in L(I)$.
This result can be strengthened as follows.

Let $N := \{n(\p)|\ \p \in L(I)\}$ be a given set of positive integers. For $0 \le i \le d$, we define 
$$c_i^N(I) :=  \sum_{\p \in L(I),\ \dim R/\p = d-i} j(I_\p)\cdot n(\p) \, .$$
The idea is to encode all local $j$-multiplicities $j(I_\p)$ in a given dimension by means of a single invariant.
For instance,  
$$c_i^N(I) =  \sum_{\p \in L(I),\ \dim R/\p = d-i} j(I_\p)$$
if $n(\p) = 1$ for all $\p \in L(I)$. Recall that $j(I_\p) = 0$ if $\p \not\in L(I)$.

\vspace{0.03in}

\begin{theorem} \label{N}
Let $R$ be an equidimensional and universally catenary Noetherian local ring of dimension $d$ and let $I \subset J$ be ideals.
The following are equivalent$\,:$
\begin{enumerate}[$(1)$]
\item $c_i^N(I)\le c_i^N(J) \quad \mbox{for \ } 0\le i \le d \, ;$
\item $c_i^N(I)= c_i^N(J) \quad \mbox{for \ } 0\le i \le d \, ;$
\item  $J$ is integral over $I .$
\end{enumerate}
\end{theorem}

\begin{proof}
The case  where $n(\p) = 1$ for all $\p \in L(I)$ was already proved by Ulrich and Validashti \cite[3.4]{UV08}.
Their proof also works in the general case. 
\end{proof}

We could not deduce Theorem \ref{N} from Theorem \ref{ID} and vice versa. It would be of interest to understand why the condition $c_i^N(I)= c_i^N(J)$ for $\, 0\le i \le d \,$ is equivalent to the condition $c_i(I) = c_i(J)$ for $0\le i \le d$. 

We now consider the case where $n(\p) = e(R/\p)$  for all $\p \in L(I)$. Define 
$$c_i^*(I) := \sum_{\p \in L(I),\ \dim R/\p = d-i} j(I_\p)\cdot e(R/\p).$$ 
The remainder of this section is devoted to the comparison between $c_i^*(I)$ and $c_i(I)$.

\vspace{0.03in}

\begin{lemma} \label{minor}
Let $R$ be a Noetherian local ring of dimension $d$ and $I$ an ideal. Then $c_i^*(I) \ge c_i(I)\, $ for $\, i \le d - \dim \, R/I$.
\end{lemma}

\begin{proof}
If $\p \in V(I)$ is a prime ideal with $\dim \, R/\p = d-i$ and $i \le d - \dim \, R/I$, then $\dim \, R/\p \ge \dim \, R/I$. 
This implies $\p \in \Min(I)$. Clearly $\Min(I) \subset L(I)$.  So we conclude that the set of primes $\p \in V(I)$ with $\, \dim \, R/\p = d-i\, $ is equal to the set of primes $\p \in L(I)$ with $\, \dim \, R/\p = d-i$.
Since every such $\p$ is in $\Min(I)$, we also have $j(I_\p) = e(I, R_\p)$.
Therefore,
$$c_i^*(I) = \sum_{\p \in V(I),\ \dim R/\p = d-i} e(I,R_\p)\cdot e(R/\p).$$ 
On the other hand according to \cite[2.3(iii) and 2.3(i)]{AM97},
$$c_i(I) = \sum_{\substack{\p\in V(I), \ \dim R/\p = d-i\\  {\rm ht}\, \p = i}} e(I,R_\p)\cdot e(R/\p).$$
\end{proof}

\begin{proposition} \label{greater}
Let $R$ be an equidimensional and universally catenary Noetherian local ring and $I$ an ideal. Then
$c_i^*(I) \le c_i(I)\, $ for $i \ge 0$, and equality holds for $\, i \le {\rm ht} \ I$.
\end{proposition}

\begin{proof}
The second statement follows from the first and Lemma~\ref{minor}. To prove the first statement, we may assume that the residue field of $R$ is 
not algebraic over a finite field, as $c_i^*$ cannot decrease (in fact stays the same) under a purely transcendental residue field extension. 
Let  $x_1,...,x_i$ be general $A$-linear combinations of a finite
generating set of $I$ as in Lemma \ref{general}, and keep in mind that by the same lemma, $x_1, \ldots, x_i$ are general elements of $I$. Write $d={\rm dim} \, R$. 
From Proposition~\ref{Lformula} we have
\begin{equation*} 
c_i(I)=\sum_{\substack{\p\in V(I), \ \dim R/\p = d-i\\
\p \supset (x_1, \ldots,  \,x_{i-1}):\, I^{\infty}}
}  \lambda \left(\frac{R_{\p}}{ (x_1, \ldots, x_{i-1})R_{\p}:I^{\infty}R_{\p} +x_iR_{\p}}\right) \cdot e(R/\p).
\end{equation*}

Consider the prime ideals $\p \in L(I)$ with $\dim R/\p = d-i$.
Since  $R$ is   equidimensional, catenary, and local, we have $\dim R_\p = {\rm ht}\, \p = d - \dim \, R/\p = i$ and therefore $j(I_\p) = c_i(I_\p)$ \cite[2.4(ii)]{AM97}.
As $L(I)$ is a finite set, $x_1,...,x_i$ are also general elements of $I_\p$ according to  Lemma \ref{general}. 
Hence, we can use $x_1,...,x_i$ to compute $j(I_\p)$ by the length formula for $c_i(I_\p)$ of Proposition \ref{Lformula}.
Notice that $j(I_\p) \neq 0$ and that $\p R_\p$ is the unique prime ideal in $R_\p$ with $\dim \, R_\p/\p R_\p = 0$. 
So we must have $\p \supset (x_1, \ldots, x_{i-1}):I^{\infty} $ and 
$$j(I_\p) = c_i(I_\p) =  \lambda \left(\frac{R_{\p}}{ (x_1, \ldots, x_{i-1})R_{\p}:I^{\infty}R_{\p} +x_iR_{\p}}\right).$$
Therefore, 
$$c_i^*(I) = \sum_{\substack{\p\in L(I), \ \dim R/\p = d-i\\
\p \supset (x_1, \ldots,  \,x_{i-1}):\, I^{\infty}}
}   \lambda \left(\frac{R_{\p}}{ (x_1, \ldots, x_{i-1})R_{\p}:I^{\infty}R_{\p} +x_iR_{\p}}\right)\cdot e(R/\p) \le c_i(I) \, .$$\end{proof}

In light of Proposition \ref{greater} we call $c_0^*(I),...,c_d^*(I)$ the {\it reduced multiplicity sequence} of $I$.
One may be tempted to ask whether $c_i^*(I) = c_i(I)$  for $0 \le i \le d$.
If this were true, it would follow directly that 
Theorem ~\ref{ID} and Theorem \ref{N} with  $n(\p) = e(R/\p)$ 
are equivalent. 
However, there are examples where $c_i^*(I) < c_i(I)$. 
We will construct such an example using St\"uckrad's and Vogel's approach to intersection theory (see \cite{FOV}), and we will explain how the multiplicity sequence and the reduced multiplicity sequence appear in the St\"uckrad-Vogel intersection algorithm.

Let $X, Y$ be equidimensional closed subschemes of $\P^n_k$,
where $k$ is an arbitrary field. In order to obtain a B\'ezout theorem for improper intersections, St\"uckrad and Vogel assigned an intersection cycle to $X \cap Y$ as follows.

Let $I_X$ and $I_Y$ denote the defining ideals of $X$ and $Y$ in $k[X_0, . . . ,X_n]$ and
$k[Y_0, . . . , Y_n]$, respectively. 
Let $k(u) := k(\{u_{i j} \, | \,  0\le i, j \le n\})$ be a purely transcendental field extension of $k$.
Consider the ring $R := k(u)[X_0, . . . , X_n, Y_0, . . . , Y_n]/(I_X, I_Y )$ and the ideal $I := (\{X_i-Y_i \, |\,0\le i \le n \})R$.
Define $x_i := \sum_{j=0}^n u_{ij}(X_j-Y_j)$, $0\le i \le n$.
Then the {\it intersection cycle} of $X$ and $Y$ is the sum of the cycles 
$$v_i := \sum_{\substack{\p\in V(I), \ {\rm ht}\, \p = i\\
\p \supset (x_1, \ldots, \, x_{i-1}):\, I^{\infty}}
}  \lambda \left(\frac{R_{\p}}{ (x_1, \ldots, x_{i-1})R_{\p}:I^{\infty}R_{\p} +x_iR_{\p}}\right) \cdot [\p]\, ,$$
where $[\p]$ denotes the cycle associated to $\p$.
By \cite[4.2]{AM97} we have
$$c_i(I) = \deg v_i = \sum_{\substack{\p\in V(I), \ {\rm ht}\, \p = i\\
\p \supset (x_1, \ldots, \, x_{i-1}):\, I^{\infty}}
}  \lambda \left(\frac{R_{\p}}{ (x_1, \ldots, x_{i-1})R_{\p}:I^{\infty}R_{\p} +x_iR_{\p}}\right) \cdot e(R/\p)\, .$$

An irreducible component $[\p]$ of the intersection cycle of $X$ and $Y$ is called $k$-{\it rational} if it is defined over $k$.  
By a result of van Gastel \cite[3.9]{Ga}, the $k$-rational irreducible components of the intersection cycle are the distinguished varieties
in Fulton's intersection theory \cite[p. 95]{Fu}. 
From the definition of $v_i$ above one sees that $[\p]$ is $k$-rational if and only if $\p \in L(I)$; for this and related results see \cite[2.2]{AM93}. 
Therefore, the proof of Proposition \ref{greater} shows that $c_i^*(I)$ is the degree of the part of the cycle $v_i$ 
that is supported at the $k$-rational components of codimension $i$.

To construct an example with $c_i^*(I) < c_i(I)$ we only need to find an example where not all $(d-i)$-dimensional components of the intersection cycle are $k$-rational.

\begin{example} 
Let $X = Y$ be the curve in $\P_k^3$ given parametrically by $(s^6: s^4t^2: s^3t^3: t^6)$, where ${\rm char} (k) \neq 2,3$. 
It was shown in \cite[Example 2, p. 269]{SV} that the intersection cycle of $X$ and $Y$ has non $k$-rational components.
From the same reference it follows that $c_3^*(I) = 11$ and $c_3(I) = 18$, where $I$ is the ideal defined above.
\end{example}

With regard to Theorem \ref{N}, it is of interest to find a practical way to compute the invariants $c_i^*(I)$. For this reason we raise the following question.

\begin{problem} 
Does there exist a bivariate polynomial such that the invariants $c_i^*(I)$, $0 \leq i \leq d$, are the normalized coefficients of its leading homogeneous component?
\end{problem}

\noindent{\bf Acknowledgment.} 
The main results of this paper were obtained at the American Institute of Mathematics (AIM) in San Jose, California, while the authors participated in a SQuaRE. We are very appreciative of the hospitality offered by AIM and by the support of the National Science Foundation. 

\bs

\end{document}